\documentclass[twoside,leqno,12pt]{amsart}

\usepackage[cp1250]{inputenc}
\date{\empty}
\usepackage{amssymb}
\usepackage{amsmath}
\theoremstyle{plain}
\newtheorem{theorem}{Theorem}[section]
\newtheorem{lemma}[theorem]{Lemma}

\newtheorem{property}[theorem]{Property}

\theoremstyle{definition}
\newtheorem{remark}[theorem]{Remark}
\newtheorem{example}[theorem]{Example}

\newcommand{\er}{\mathbb{R}}

\newcommand{\en}{\mathbb{N}}

\def\grad{\nabla}

\begin{document}
\title[Bifurcation values of $C^{\infty}$ functions ]{Bifurcation values of $C^{\infty}$ functions}
\author[Michał Klepczarek]{Michał Klepczarek}
\thanks{This research was partially supported by OPUS Grant No 2012/07/B/ST1/03293 (Poland)}
\date{\today}

\begin{abstract} 

We show how one can use a trivialization of a function $f:\er^n\to\er$ on fibers of some function $g:\er^n\to\er$ to construct a trivialization of $f$ in $\er^n$. Additionally we adopt a method  for trivialising functions which satisfy the $\rho_0$-regularity condition to the case of functions defined on hypersurfaces of the form $M=g^{-1}(0)$.

 
\end{abstract}
\maketitle

\section*{Introduction}

Let $f:\er^n\to\er$ be a polynomial. It is well-known that there exists a finite set such that $f$ is a $C^{\infty}$ fibration over the complement of this set. The smallest such a set is called the set of \textit{bifurcation values} of $f$ and denoted by $B(f)$. It contains the set of critical values $K_0 (f)$ and the set of \textit{bifurcation values at infinity} $B_{\infty} (f)$ of $f$. Finding an effective description of the set $B_{\infty} (f)$ is still an open question. However we can approximate $B_{\infty}(f)$ using different supersets. The most popular one uses the so-called Malgrange condition. 

We say that $f$ satisfies the \textit{Malgrange condition} in $\lambda\in \er$ if there exists a neighborhood $U$ of $\lambda$ and constans $\delta,R >0$ such that
 \begin{equation}\tag{M} \label{M}
 \forall_{x\in f^{-1} (U)} \quad \|\grad f(x) \| \|x\| \geqslant \delta \quad \text{for} \quad \|x\|\geqslant R.
\end{equation}  
The set $K_{\infty}(f)$ of all values which do not satisfy Malgrane's condition is called the set of \textit{asymptotic critical values} of $f$ i.e.
$$
K_{\infty}(f):=\{\lambda\in\er  | \exists_{(x_k)\subset\er^n} \|x_k\|\to\infty , f(x_k)\to\lambda , \|x_k\| \|\grad f(x_k)\|\to 0 \}.
$$
Jelonek and Kurdyka in \cite{JK1} gave an effective characterization of this set. Moreover in \cite{JK2} they showed how to proceed when the polynomial $f$ is defined on an algebraic set (see also \cite{Je}).
 
In this paper we show how one can use a trivialization of a smooth function $f:\er^n\to\er$ on fibers of some other smooth function $g:\er^n\to\er$ to construct a desired
trivialization of $f$ in $\er^n$. More precisely, we introduce a notion of $(g,S)$-Malgrange condition; for a smooth function $g$ and an open set $S\subset \er$ we say that  $f$ satisfies the $(g,S)$-\emph{Malgrange condition} in $\lambda$ if

 \begin{enumerate}
  \item $\grad g(x)\neq 0$ near $f^{-1}(\lambda)$ outside some compact set
  \item $g^{-1}(S)$ contains all fibers of $f$ near $\lambda$
  \item for any $s\in S$ the function $f|_{g^{-1}(s)} : g^{-1}(s)\to\er$ satisfies the Malgrange condition in $\lambda$
\end{enumerate}
(compare to the definition in Section 3).
The set of all $\lambda$ that don't satisfy the $(g,S)$-Malgrange condition we will denote by $K_{\infty} ^{(g,S)}(f)$. Our aim is to prove $B(f)\subset  K_{\infty} ^{(g,S)}(f)\cup K_0 (f)$ and therefore 
$$
B(f)\subset \bigcap_{(g,S)} K_{\infty} ^{(g,S)}\cup K_0 (f)
$$
(see Theorem \ref{tw.tryw foliacji} and Remark \ref{remark (g,S)-Mal}). It is worth noting that in (3) of the definition of $(g,S)$-Malgrange condition we allow different constans $\delta_s$ (in inequality(\ref{M})) for different $s\in S$. Therefore our method may be applied even if $\inf \{\delta_s |\;s\in S\}=0$ (i.e. when the standart Malgrange condition fails). This often leads us to more sharp aproximation of $B(f)$ than in the classical method (see Example \ref{exa2new} and Example \ref{exa3new}).

The other popular approach to finding the bifurcation values uses a critical set $\mathcal{M}_a(f)$ of the map $(f,\rho_a)$ where $\rho_a$ is the Euclidean distance function from a fixed point $a\in\er^n$. We can define the set of asymptotic $\rho_a$-nonregular values as
$$
S_a(f):=\{\lambda\in\er |\; \exists_{(x_k)\subset \mathcal{M}_a(f)} \|x_k\|\to\infty , f(x_k)\to\lambda \}.
$$
It has been proven in \cite{Ti},\cite{DRT},\cite{DT} that $B_{\infty}(f)\subset S_a(f)$ for any $a \in \er^n$, thus in particular 
$$
B_{\infty}(f)\subset S_{\infty} (f):=\bigcap_{a\in\er^n} S_a(f).
$$
In the second part of this paper we will show how to get a simillar result when $f$ is defined on the manifold of the form $g^{-1}(0)$. We will introduce an analog of the condition used in \cite{NZ1} and \cite{NZ2} that allows us to trivialize function $f$ (see Theorem \ref{tw. trywializacja war. N} and Remark \ref{rho}).

\section{Auxiliary results}

In this section we collect some useful facts about differential equations, which we will use later in this article.

Let $M$ be a smooth $m$-dimantional manifold. We denote by $T_x M$ the tangent space to the manifold $M$ at a point $x\in M$ and $TM:=\bigcup_{x\in M} T_x M$. Let $W:M\to TM$ be a smooth vector field on $M$.

\begin{lemma}\label{lem.oprzedlrozw}
Let $\phi :(t_0,\xi ]\to M$ be a solution of the system of diffe\-rential equations $x'=W(x)$. Assume that there exists a sequence  $t_k \in (t_0,\xi)$, $k\in\en$, 
 such that $\lim_{k\to\infty} t_k = t_0 $ and $\lim_{k\to\infty} \phi(t_k)=x_0 \in M $. Then there exists $lim_{t\to t_0}\phi (t)=x_0$ and $\phi $ is not the maximal solution to the left.
\end{lemma}

\begin{proof}
Let $\varphi=(\varphi_1,\ldots,\varphi_m):U\to A\subset\er^m$ be a map in a neighborhood $U\subset M$ of the point $x_0$. Choosing a subsequence if necessary, we may assume that $x_k:=\phi (t_k)\in U$ for $k\in\en$. Than we can write 
$$
W(x)=\sum_{i=1}^{m} W_i (x)\frac{\partial}{\partial \varphi_i} \qquad \text{ for } x\in U, $$
where $W_i:U\to\er$ for $i=1,..,m$. 

Let $y_0:=\varphi (x_0)$ and $y_k:=\varphi (x_k)$ for $k\in\en$. 

Let $\alpha\in\er\cup\{-\infty\}$ be a minimal number for which $\phi((\alpha,t_1])\subset U$. Put $I=(\alpha,t_1]$. 
Denote 
$\phi_\varphi := \varphi\circ \phi :I\to A$ and  
$\phi_{\varphi _i}:=\varphi_i\circ \phi $
 for $i=1,..,m$. Then we have $\phi'(t)=\sum_{i=1}^{m} \phi_{\varphi_i}'(t)\dfrac{\partial}{\partial\varphi_i}$ for $t\in I$ and 
 $$\sum_{i=1}^{m} \phi_{\varphi_i}'(t)\dfrac{\partial}{\partial\varphi_i}=\sum_{i=1}^{m} W_i (\phi(t))\frac{\partial}{\partial \varphi_i}.
 $$ 
Therefore 
 \begin{equation}\label{eq1}
 \qquad \phi_{\varphi_i}'(t)=W_{\varphi^{-1}}^i (\phi_{\varphi_i}(t)) \qquad \text{for } t\in I \;\text{ } i=1,..,m,
 \end{equation}
 where $W_{\varphi^{-1}}^i (y):=W_i \circ \varphi ^{-1} (y)$ for $y\in A$. 

To complete the proof we need the following well known fact. 

\begin{property}\label{Property1} There exists an interval $J$ such that $t_0\in J$ and a neighborhood $\Gamma\subset \er\times A$ of $(t_0,y_0)$ such that for any $(t',y')\in\Gamma$ every maximal solution $\gamma$ of sytem (\ref{eq1}) that passes through $(t',y')\in\Gamma$ is defined at least on $J$ and the graph of $\gamma|_{J}$ is contained in some rectangle $T\subset \er\times A$. 
\end{property}

By choosing a subsequence of the sequence $(t_k,y_k)$, we may assume that $(t_k,y_k)\in \Gamma$ for $k\in\en$ and $\xi>t_k>t_l>t_0$ for $k<l$. 

From Property \ref{Property1} we have that $\alpha \le t_0$. Indeed, otherwise there exists $t'\in  (t_0,t_1)$ such that
  $\phi (t')\notin U$, hence there exists a  maximal solution to the left $\widehat{\phi _\varphi}:\widehat{I}\to A$ of the system (\ref{eq1}) such that $\widehat{\phi _\varphi}$ goes through
$\Gamma$ and $t_0\notin \widehat{I}$ which contradicts Property \ref{Property1}. 

Therefore $(t_0,t_1]\subset I$ and $\phi_\varphi={\phi^{*}_\varphi}|_{I}$ where $\phi^{*}_\varphi:I^{*}\to A$ is a maximal solution to the left of the system (\ref{eq1}) and $t_0\in I^{*}$. Consequently 
$$
\lim_{t\to t_0} \phi (t)=\lim_{t\to t_0}\varphi^{-1}(\phi^{*}_{\varphi} (t))=\lim_{t\to t_0}\varphi^{-1}(y_0)=x_0,
$$
which completes the proof of Lemma \ref{lem.oprzedlrozw}.
\end{proof}

\begin{lemma}\label{lem.uciekanie}

Let  $\phi :(\alpha,\beta )\to M$ be a  maximal solution of the system
 $x'=W(x)$. For every compact set $K\subset \er\times M$ there exist $\alpha^{*},\beta^{*}\in\er$ such that the graphs of $\phi|_{(\alpha,\alpha^{*})}$ and $\phi|_{(\beta^{*},\beta)}$ are disjoint with $K$.

\end{lemma}

\begin{proof}

Let $A:=\{ t\in (\alpha , \beta) |\; (t,\phi (t))\in K\}$. If $A=\emptyset$ then as  $\alpha^{*} ,\beta^{*}$ we may choose arbitrary numbers from $(\alpha, \beta)$. Now assume that $A\neq \emptyset$ and let $\alpha^{*}=\inf A$, $\beta^{*}=\sup A$. Observe that  $\alpha <\alpha^{*}$ and $\beta^{*}<\beta$. Indeed, if $\alpha =\alpha^{*}$ than there exists a sequence $(t_k)_{k=1}^\infty$ such that $(t_k,\phi(t_k))$ converges to a poin in $K\subset \er\times M$. This contradicts Lemma \ref{lem.oprzedlrozw}. Analogously we prove  that $\beta^{*}<\beta$. 
\end{proof}

\section{The Malgrange condition on Manifolds}

Let $f,g\in C^{\infty} (\er^n)$. We will assume that $ \grad g(x)\neq 0$ for $x\in g^{-1} (0)$. Denote $M:=V(g)=g^{-1} (0) \text { and } f_M :=f|_{M}$.  
Consider the following vector field $\grad f_M: M\to \er^n$
\begin{equation}\label{def.grad f_M}
\grad f_M(x):=\grad f(x)-\dfrac{\langle\grad f(x),\grad g(x)\rangle}{||\grad g(x)||^{2}}\grad g(x),\quad x\in M.
\end{equation}
Geometrically  $\grad f_M(x)$ is the projection of $\grad f(x)$ onto the tangent space $T_x M$. 

A value $\lambda\in \er$ is called a \emph{regular value} of  $f_M$ if $\grad f_M(x)\neq 0$ for $x\in f_M ^{-1} (\lambda).$
The set of all values $\lambda$ that are not regular we will denote by $K_0(f_M)$.

We say that $f_M$ satisfies the \emph{Malgrange condition} 
 over $U\subset \er$  if there exist  constans $\delta,R >0$ such that
 \begin{equation*}\label{war.Malg}
 \forall_{x\in f_{M}^{-1} (U)} \quad \|\grad f_M(x) \| \|x\| \geqslant \delta \quad \text{for} \quad \|x\|\geqslant R.
\end{equation*}

We say that $f_M$ satisfies the \emph{Malgrange condition} 
in $\lambda\in\er$ if there exists a neighborhood $U$ of $\lambda$ such that $f_M$ satisfies the Malgrange condition over $U$.
We will denote by $K_{\infty} (f_M)$ the set of all $\lambda$ that do not satisfy the Malgrange condition.

It is well known that the Malgrange condition allows us to integrate $\grad f_M(x)/  \|\grad f_M(x)\|^{2} $ field and get the trivialization of $f_M$ (see \cite{Ra} and \cite{Je}). More precisely, we have

\begin{theorem}\label{tw.trywializacja przy war.M}
Let 
 $\lambda\in\er$ be a regular value of $f_M$. If $f_M$ satisfies the Malgrange condition in $\lambda\in\er$ than there exists a neighborhood $U$ of $\lambda$ such that ${f_M}_{|f_M ^{-1}(U)}$
is a $C^{\infty}$ trivial fibration  on $M$. 
\end{theorem}

Immediately from Theorem \ref{tw.trywializacja przy war.M} we get

\begin{remark} 
$B(f_M)\subset K_{\infty}(f_M)\cup K_0 (f_M)$.
\end{remark}

Note that if $f_M$ satisfies the Malgrange condition, it does not follow that $f|_{M^\varepsilon}$ satisfies it also, 
 where $M^\varepsilon=g^{-1}((-\varepsilon,\varepsilon))$. In other words, one can not trivialise $f_M$ using field $\grad f_M $ in a neighbourhood $M^\varepsilon$ 
 of $M$ as it is shown by the following example.

\begin{example}\label{ex1}
Let $f,g\in\er[x,y]$ be polynomials defined as $g(x,y):=y$, $f(x,y):=x-x^3 y^2$. Put $M=g^{-1}(0)$. Obviously  $\grad f_M=[1,0]$ and $f$ satisfies Malgrange condition on $M$. Consequently $f_M$ is a trivial fibration on $M$. On the other hand, for any $\varepsilon>0$, denoting  $ M_{t}:=g^{-1}(t) $ for $t\in (-\varepsilon,\varepsilon)$  we have $\grad f_{M_{g(x,y)}}(x,y)=(1-3x^2 y^2,0)$ for $(x,y)\in M^{\varepsilon}$. Moreover, for
 $$
 (x_n,y_n)=\left(\dfrac{n\varepsilon}{2},\dfrac{2}{\sqrt{3}n\varepsilon}\right),\quad n\in \en
 $$
we have $\|\grad f_{M_{g(x,y)}}(x_n,y_n)\|=0$ for $n \in \en$ and $\lim_{n\to\infty} \|(x_n,y_n)\|=\infty$, and $(x_n,y_n)\in M^\varepsilon$ as $n\to\infty$. So we can not use the trivialization method without restricting the field $\grad f_M$ to the manifold $g^{-1}(0)$. 
\end{example}
\section{The Malgrange condition on fibers}

In this section we will present a method of using fibers of some function $g$ to construct a trivialisation of a function $f$.

Let $ f,g\in C^{\infty} (\er^n)$. Let $D\subset \er^n$ be an open set and $\grad g(x)\neq 0$ for $x\in D$.
As in (\ref{def.grad f_M}) we can define $\grad f_{g^{-1}(g(x))}(x)$ for $x\in D$. Put
 $$
\grad_g f(x):=\grad f_{g^{-1}(g(x))}(x)\quad \hbox{for }x\in D.
$$
 
Geometrically  $\grad_g f(x)$ is the projection of $\grad f(x)$
onto the tangent space at $x$ of the fibre $g^{-1}(g(x))$.


Denote $\Xi=\{(g,S) |\; g\in C^{\infty}(\er^n), S \text{ open in } \er \}$ and let $(g,S)\in \Xi$.
We say that $f$ satisfies the \textit{$(g,S)$-Malgrange condition} 
 over $U$  if there exists a constant $R >0$ such that 
 \begin{enumerate}
  \item $\grad g(x)\neq 0$ on $D_{(U,R)}:=f^{-1}(U)\setminus \{x\in \er^n|\;R\geqslant\|x\|\}$
  \item $D_{(U,R)}\subset g^{-1}(S)$
  \item $
 \forall_{s\in S}\;\exists_{\delta_s >0}\; \|\grad_g f(x) \| \|x\| \geqslant \delta_s \quad \text{for} \quad x\in D_{(U,R)}\cap g^{-1} (s).
 $
\end{enumerate}

We say that $f$ satisfies the \textit{$(g,S)$-Malgrange condition} 
in $\lambda\in\er$ if there exists neighborhood $U$ of $\lambda$ such that  $f$ satisfies the $(g,S)$-Malgrange condition over $U$. We will denote by $K_{\infty} ^{(g,S)}(f)$ the set of all $\lambda$ that don't satisfy the $(g,S)$-Malgrange condition.

The main result of this section is the following

\begin{theorem}\label{tw.tryw foliacji}
Let $\lambda$ be a regular value of $f$. If $f$ satisfies the $(g,S)$-Malgrange condition in $\lambda\in\er$ than there exists a neighborhood $U$ of $\lambda$ such that ${f}|_{f ^{-1}(U)}$
is a trivial fibration. 
\end{theorem}

The proof of the above theorem will be preceded by two properties and a lemma. 

From now we will assume that $\lambda\in \er$ is a regular value of $f$ and that $f$ satisfies the $(g,S)$-Malgrange condition in $\lambda $. 

The following property holds


\begin{property}\label{lem. o braku osobliwosci 2}
There exists a  neighborhood $U$ of $\lambda$ such that $\grad f(x)\neq 0$ for $x\in f ^{-1} (U)$.
\end{property}

Let  $U,R$ be as in the $(g,S)$-Malgrange condition. Shrinking the set $U$ if necessary, we can assume that $\grad f(x)\neq 0$ for $x\in f ^{-1} (U)$. Let $\alpha,\beta $ be $C^{\infty}$ functions in $\er^n$ such that 

$$
\alpha(x) = \left\{ \begin{array}{ll}
0 & \textrm{for $\|x\| \geqslant R+1$}\\
1 & \textrm{for $\|x\| \leqslant R $}
\end{array} \right.
$$

$$
\beta(x) = \left\{ \begin{array}{ll}
1 & \textrm{for $\|x\| \geqslant R+1$}\\
0 & \textrm{for $\|x\| \leqslant R $}
\end{array} \right.
$$
and $0< \alpha(x),\beta (x)< 1$ for $\|x\|\in (R,R+1)$.

 We define a smooth vector field $w: f ^{-1} (U) \to \er^n$ as
$$w(x):=\alpha(x)\grad f(x)+\beta(x)\grad_g f(x).$$
Here we are using convention that $\beta(x)\grad_g f(x)=0$ for $\|x\| \leqslant R $ (note that $\grad_g f(x)$ might not be defined for some points $x$ such that $\|x\|\le R$).

From the definition of $w$ and Property \ref{lem. o braku osobliwosci 2} we get 
\begin{property} \label{prop. pola w}
Under the above assumptions we have
\begin{enumerate}
\item[(i)] $w(x)=\grad f(x)$ for $\|x\|\leqslant R$ and $w(x)=\grad_g f(x)$ for $\|x\|\geqslant R+1$ 
\item[(ii)] $\langle w(x),\grad f(x) \rangle\neq 0$ for $x \in f ^{-1} (U).$
\end{enumerate}
\end{property}

Let us define $u:f ^{-1} (U)\to\er^n$ as 
$$u(x):=\dfrac{w(x)}{\langle w(x),\grad f(x) \rangle}.$$
From the assumptions and Property \ref{prop. pola w} we see that  $u$ is well defined, it is smooth and $u(x)\neq 0$ for $x\in f^{-1}(U)$.
\begin{lemma}\label{lem. wlasnosci u}
Under the above assumptions we have
\begin{enumerate}

  \item[(i)] $\langle\grad f(x),u(x)\rangle=1$ for $x\in f^{-1}(U)$
  \item[(ii)] For any $s\in S$ there exists a constant $\alpha_s >0$ such that 
  $$|\langle x,u(x)\rangle|\leqslant \varepsilon_s \|x\|^{2} \qquad \text{ for } x\in f^{-1}(U)\cap g^{-1}(s),\|x\|>R+1.$$
\end{enumerate}
\end{lemma}
\begin{proof} (i) is obvious. We will prove (ii).  From the $(g,S)$-Malgrange condition we have $\dfrac{1}{\|\grad_g f \|}\leqslant \dfrac{1}{\delta_s} \|x\|$\quad for 
  $x\in f^{-1}(U)\cap g^{-1}(s)$, $\|x\|>R$. Since $\langle \grad_g f(x),\grad f(x)\rangle=\|\grad_g f(x)\|^2$, using Schwartz inequality we get   
   $$|\langle x,u(x) \rangle|\leqslant \|x\| \|u(x)\|=\|x\| \dfrac{1}{\|\grad_g f (x)\|}\leqslant \dfrac{1}{\delta_s} \|x\|^2 \text{ for } \|x\|>R+1,$$    
 which completes the proof.   
\end{proof}

We are ready to prove Theorem \ref{tw.tryw foliacji}.

\begin{proof}
Let $U,R,u$ be defined as above.
For each $\mu \in U$, consider a system of differential equations
\begin{equation}\label{system mu3}
x'= (\lambda - \mu)u(x)
\end{equation}
with the right side defined in 
$G:=\{(t,x)\in \er\times D|\; x\in f^{-1}(U)\}.$
\linebreak
Denote by $\Phi_\mu : V_\mu \to D $ the general solution of system (\ref{system mu3}) and put
\linebreak $V_\mu :=\{ (\tau,\eta,t) \in \er\times D\times R|\; (\tau,\eta) \in G , t\in I_\mu (\tau,\eta)\}$, where $I_\mu (\tau,\eta)$ is a domain of integral solution of $ t\to\Phi_\mu (\tau,\eta,t)$. From the definition of the general solution  we get 
\begin{equation}\label{eq. war poczatkowy}
\Phi_\mu (\tau,\eta,\tau)=\eta.
\end{equation}
Consider the mapping 
$$ \Psi_1: f^{-1}(U) \ni x \mapsto \Phi_{f(x)} (0,x,1)\in f^{-1} (\lambda).$$

We show that the mapping $ \Psi_1$ is well defined i.e. $ 1\in I_{f(x)} (0,x)$ for each $x\in f^{-1}(U)$. 
Suppose the contrary that there exists $x\in f^{-1}(U)$ such that $1\notin I_{f(x)} (0,x)$. Then the right end-point $\beta$   of the interval $I_{f(x)} (0,x)$ satisfies $0<\beta \leqslant 1$.  Let $\varphi_x$ be an integral solution of the system (\ref{system mu3}) with $\mu = f(x)$ satisfying the initial condition $\varphi_x (0)=x$, that is
\begin{equation}\label{eq. war pocz varphi2}
\varphi (t)= \Phi_{f(x)} (0,x,t) \quad \text{ for } t\in I_{f(x)} (0,x).
\end{equation}
We have
 $$ 
 (f \circ \varphi_x )'(t)= \langle \grad f (\varphi _x (t) ), \varphi _x '(t) \rangle=
 \lambda-f(x)\quad \text{for} \quad t \in I_{f(x)} (0,x).$$
Therefore
\begin{equation}\label{eq. (*)}
f \circ \varphi_x (t)=(\lambda-f(x))t+f(x),\quad t \in I_{f(x)} (0,x)
\end{equation}
and $f\circ \varphi_x (t)\in J$ for $t\in [0,\beta)$, where $J$ is a closed interval with endpoints $\lambda$ and $f(x).$

Denote 
$$
K:=\{(t,x')\in \er\times f ^{-1}(U) |\; t\in [0,1], f (x') \in J, \|x'\|\leqslant R+1\}
.$$
 Obviously $K$ is a compact set.   Lemma \ref{lem.uciekanie} implies that 
there exists $\tau \in (0,\beta)$ such that $(t,\varphi _x (t)) \notin K$ for $t\in [\tau,\beta)$.  
Since $J\subset U$, we have $\|\varphi _x (t)\| >R+1$ for $t\in [\tau,\beta)$.

Consider a function $\varrho :[\tau,\beta)\to\er$ defined by
$$\varrho (t):=\dfrac{1}{2} \text{ln} \|\varphi _x (t)\|^{2}.$$
Let $s_0\in S$ be such that $\varphi _x (\tau)\in D_{s_0}$. Using Lemma \ref{lem. wlasnosci u} (ii) we get
$$|\varrho '(t)|=\dfrac{|\langle \varphi _x (t),\varphi_x '(t)\rangle|}{\|\varphi _x (t)\|^{2}}=\dfrac{|\lambda-f(x)|}{\|\varphi _x (t)\|^{2}}|\langle \varphi_x (t),u(\varphi_x (t))\rangle|\leqslant \varepsilon_{s_0}|\lambda-f(x)|$$
for $t\in (\tau,\beta)$. 
From the mean value theorem there exists $\theta_t \in (\tau,t)$ such that $\varrho (t)- \varrho(\tau)=\varrho '(\theta_t) (t-\tau)$ and therefore, by the above,
$$\varrho (t) \leqslant \varrho(\tau)+\varepsilon_{s_0}|\lambda-f(x)|(t-\tau)\leqslant \varrho(\tau)+\varepsilon_{s_0}|\lambda-f(x)|(\beta-\tau).
$$
Denoting $L:=\varrho(\tau)+\varepsilon_{s_0}|\lambda-f(x)|(\beta-\tau)$ 
we see that the solution ${\varphi_x} _{|(\tau,\beta)}$ is
contained in the compact set 
$$
\{(t,x')\in \er\times f^{-1}(U) |\; f(x')\in J, \|x'\|\leqslant \text{e} ^{L} \}\subset \er\times f^{-1}(U)
,$$
 which contradicts Lemma \ref{lem.uciekanie}. 

Summing up we have shown that $1\in I_{f(x)} (0,x)$ for every $x\in f^{-1}(U)$. Then from (\ref{eq. (*)}) we get  $f(\Psi_1(x))=f(\varphi_x(1))=\lambda$ and the mapping $\Psi_1$ is defined correctly.
  Similarly as above we show that the mapping 
  $$\Theta: f ^{-1} (\lambda))\times U \ni (\xi,\mu)\mapsto \Phi_\mu (1,\xi,0)\in f^{-1}(U)$$
is also well defined. It is easy to check that the mapping  
$$\Psi: f^{-1}(U) \ni x \mapsto (\Psi_1 (x),f(x))\in f ^{-1} (\lambda)\times U$$
is a $C^{\infty}$ diffeomorphism and $\Psi^{-1} =\Theta$. Therefore $f|_{f^{-1}(U)}$
is a  trivial fibration. 
\end{proof}
Immediately from Theorem \ref{tw.tryw foliacji} we get

\begin{remark} \label{remark (g,S)-Mal}
For every $(g,S)\in \Xi$ we have $B(f)\subset K_{\infty} ^{(g,S)}(f)\cup K_0 (f)$. Therefore $
B(f)\subset \bigcap_{(g,S)\in \Xi} K_{\infty} ^{(g,S)}\cup K_0 (f)
.$

\end{remark}
In Theorem~~\ref{tw.tryw foliacji} we showed how to trivialize a function $f$ using fibers of some function $g$. Now we give some examples where the assumptions of Theorem 1 are not met but using Theorem~~\ref{tw.tryw foliacji} we can deduce the triviality of $f$.

\begin{example}\label{exa2new}  
Let $f:\er^2\to \er$
$$ f(x,y):=\dfrac{y}{1+x^{2}}.$$
 It is easy to check that $f$ does not satisfy the Malgrange condition in 0. Denote $g:\er^2\to\er$
 $$ g(x,y):=x,\quad S=\er.$$
 Put $U=\er$ and $R=1$. Then
 \begin{enumerate}
  \item $\grad g(x)=[1,0]\neq 0$ for $D_{(U,R)}=\er^2\setminus \{x\in \er^2|\;R\geqslant\|x\|\}$
  \item $D_{(U,R)}\subset g^{-1}(\er)$
  \item for $s\in \er$ we have
$$\|\grad_g f(x,y) \|=\|[0,\dfrac{1}{1+x^2} ]\|= \dfrac{1}{1+s^2} \qquad \text{ for } (x,y)\in D_{(U,R)}\cap g^{-1}(s) .$$
\end{enumerate}
  
Therefore $f$ satisfies the $(g,S)$-Malgrange condition 
and using Theorem~~\ref{tw.tryw foliacji} we deduce that $f$ is a trivial fibration.
\end{example}

An example of polynomials which is a trivial fibration but does not satisfy the Malgrange condition comes from L. P\u aunescu and A. Zaharia (see \cite{PZ}).

\begin{example}\label{exa3new} 
Let $p,g\in\en,f:\er^3\to\er$
$$f(x,y,z):=x-3x^{2p+1}y^{2q}+2x^{3p+1}y^{3q}+yz.$$
L. P\u aunescu and A. Zaharia showed that after a suitable polynomial change of coordinates, we  can write $f(X,y,Z)=X$.
 Therefore $f$ is a trivial fibration. Following their reasoning, we can deduce that if $p>q$ then $f$ does not satisfy the Malgrange condition in 0.
Let 
 $$ g(x,y,z):=y \quad\text{ for } (x,y,z)\in \er^3, S=\er.$$
Put $U=\er,R=1$. Then
\begin{enumerate}
  \item $\grad g(x,y,z) =[0,1,0]\neq 0 $ for $D_{(U,R)}=\er^3\setminus \{x\in \er^3|\;R\geqslant\|x\|\}$
  \item $D_{(U,R)}\subset g^{-1}(\er)$
  \item for $s\in \er$ we have
\begin{equation*}
\begin{split}
\grad_g f(x,y,z) &=[1-3(2p+1)x^{2p}y^{2q}+2(3p+1)x^{3p}y^{3q},0,y]\\
&=[1-3(2p+1)x^{2p}s^{2q}+2(3p+1)x^{3p}s^{3q},0,s] 
\end{split}
\end{equation*}
for $(x,y)\in D_{(U,R)}\cap g^{-1}(s)$. If $s\neq 0$ then $\|\grad_g f(x,y,z)\|\geq\|s\|>0$ and if $s=0$ we have $\|\grad_g f(x,y,z)\|=1.$
\end{enumerate}
Therefore $f$ satisfies the $(g,S)$-Malgrange condition 
and using Theorem \ref{tw.tryw foliacji} we deduce that $f$ is a trivial fibration.\\
\end{example}

In general, finding a suitable function $g$ can be very difficult. In the case when $f$ is a coordinate of a mapping with non-vanishing jacobian, the natural candidates for $g$ are other coordinates of this mapping.

\begin{example}\label{exa4new} 
Let $F=(f_1,f_2):\er^2\to\er^2$ and $Jac(F)=1$, where $Jac(F)$ is the jacobian of $F$. Than we have
$$
\|\grad_{f_2} f_1 (x,y)\|^2=\frac{Jac(F)^2}{\|\grad f_2(x,y)\|^2}=\frac{1}{\|\grad f_2(x,y)\|^2},\quad (x,y)\in\er^2.
$$
Therefore $f_1$ satisfies the $(f_2,\er)$-Malgrange condition over $\er$ if and only if there exists $R>0$ such that
$$
\|(x,y)\|>\delta_s \|\grad f_2(x,y)\|
\quad \text{for } (x,y)\in f_2 ^{-1} (s),\|(x,y)\|>R,\;s\in\er.
$$

\end{example}

From Example \ref{exa4new} we get

\begin{remark}
Let $F=(f_1,f_2):\er^2\to\er^2$ be a smooth mapping with $Jac(F)=1$ and $f_2$ be a polynomial, such that $\deg f_2\le 2$.
Then the mapping $F$ is injective.
\end{remark}

\section{$\rho_0$-regularity on manifolds}

In this section we will consider a different condition that allows us to trivialize functions defined on the manifold of the form $g^{-1}(0)$ at infinity. 

Let $f,g\in C^{\infty} (\er^n)$. Assume that $\grad g(x) \neq 0$ for $x\in g^{-1}(0)$ and denote $M:=g^{-1}(0)$, $\rho_0:=\|\cdot\|^{2}|_M$. 

The critical set $\mathcal{M}_0(f_M)$ of the map $(f_M,\rho_0):M\to\er$ we will call \textit{the Milnor set of $f_M$} (with respect to $\rho_0$ function).

A value $\lambda\in \er$ is called a \emph{$\rho_0$-regular value of  $f_M$ at infinity} if there exist a neighborhood $U$ of $\lambda$ and a constant $R>0$ such that

 \begin{equation*}\label{rho_0}
 \forall_{x\in f_{M}^{-1} (U)} \quad x\notin \mathcal{M}_0(f_M) \quad \text{for} \quad \|x\|\geqslant R.
\end{equation*}  
The set $S_0(f_M)$ of all values that are not a $\rho_0$ regular value of $f_M$ at infinity will be called the set of \textit{asymptotic $\rho_0$-nonregular values} of $f_M$ i.e. 
$$
S_0(f_M):=\{\lambda\in\er |\; \exists_{(x_k)\subset \mathcal{M}_0(f_M)} \|x_k\|\to\infty , f_M(x_k)\to\lambda \}.
$$

Our aim is to prove the following theorem
\begin{theorem}\label{tw. trywializacja war. N}
If $\lambda$ is a $\rho_0$-regular value of $f_M$ at infinity than  there exists neighborhood $U$  of the $\lambda$ and $R>0$  such that ${f_M}| _{f_M ^{-1}(U)}$ is a $C^{\infty}$ trivial fibration  on $M$ at infinity.
\end{theorem}

The proof of the theorem will be preceded by some technical properties.
 
Let $M^{*}:=M\setminus \mathcal{M}_0(f_M)$ and consider
the vector field $v:M^* \to \er^n$
\[
\begin{split}
v(x)&:=\grad f(x)+\dfrac{\langle\grad g(x),x\rangle\langle\grad g(x),\grad f(x)\rangle-\|\grad g(x)\|^2 \langle x,\grad f(x)\rangle}{\|x\|^2 \|\grad g(x)\|^2-\langle\grad g(x),x\rangle^2}x \\
&+\frac{\langle\grad g(x),x\rangle\langle x,\grad f(x)\rangle -\|x\|^2 \langle\grad g(x),\grad f(x)\rangle}{\|x\|^2 \|\grad g(x)\|^2-\langle\grad g(x),x\rangle^2}\grad g(x) \text{ for } x\in M^*.
\end{split}
\]
The field $v$ is well defined. Indeed if $\|x\| \|\grad g(x)\|=|\langle\grad g(x),x\rangle|$ then using the Cauchy–Schwarz inequality we deduce that $x$ and $\grad g(x)$ are linearly dependent. We get $\grad \rho_0(x)=0$ and therofore $x\in\mathcal{M}_0(f_M)$ which contradicts the assumptions.

From the definition of $v$ we see that $v(x)$ is tangent to the manifold $M^*$ and to the sphere $\partial B(x):=\{y\in \er^n | \quad \|y\|=\|x\|\}$. Namely, we have

\begin{property}\label{propertyofv}
For any $x\in M^*$ we have  $\langle v(x),x\rangle =\langle v(x),\grad g(x)\rangle=0$.
\end{property}

\begin{proof} Take $x\in M^*$. Then
\[
\begin{split}
\langle v(x),x\rangle=&\langle \grad f(x),x\rangle\\
&+\dfrac{-\|\grad g(x)\|^2 \langle x,\grad f(x)\rangle\|x\|^2+\langle\grad g(x),x\rangle^2\langle x,\grad f(x)\rangle
}{\|x\|^2 \|\grad g(x)\|^2-\langle\grad g(x),x\rangle^2}, 
\end{split}
\]
therefore
\[
\begin{split}
\langle v(x),x\rangle=&\langle \grad f(x),x\rangle\\
&+\langle \grad f(x),x\rangle\dfrac{-\|\grad g(x)\|^2 \|x\|^2+\langle\grad g(x),x\rangle^2
}{\|x\|^2 \|\grad g(x)\|^2-\langle\grad g(x),x\rangle^2}=0 ,
\end{split}
\]
which gives that $\langle v(x),x\rangle=0$. 
Analogously as above we obtain that $\langle v(x),\grad g(x)\rangle=0$.
\end{proof}


\begin{lemma}\label{lem. wlasnosci v}
For $  x\in M^*$  we have $\langle v(x),\grad f_M(x)\rangle\neq0 .$

\end{lemma}

\begin{proof}
Denote 
$$
\Omega :=\{w:M^*\to\er^n |\; w- smooth\},
$$
$$\Omega(TM^*):=\{w\in\Omega|\; \forall_{x\in {M}^*}\; \langle w(x),\grad g(x)\rangle =0\}
$$
and let $\pi_{{M}^*} :\Omega\to \Omega(TM^*)$ be a mapping defined by
$$
  \pi_{{M}^*} (w(x)):=w(x)-\frac{\langle w(x),\grad g(x)\rangle}{\|\grad g(x)\|^2}\grad g(x) \quad \text{for } x\in M^*.
$$
At first we will prove that 
\begin{equation}\label{eq. v}
v(x)=\pi_{(\pi_{{M}^*}(x))^{\perp}}(\grad f_M(x)) \text{ for } x\in M^*,
\end{equation}
 where
$$
\pi_{(\pi_{{M}^*}(x))^{\perp}} (w(x)):=w(x)-\frac{\langle w(x),\pi_{{M}^*}(x)\rangle}{\|\pi_{{M}^*}(x)\|^2}\pi_{{M}^*} (x)
 \text{ for } x\in M^*,w\in\Omega.
$$

From definitions and simple calculations for $x\in M^*$ we have 
$$
\pi_{(\pi_{{M}^*}(x))^{\perp}}(\grad f_M(x))=\grad f_M(x)-\frac{\langle \grad f_M(x),\pi_{M^*}(x)\rangle}{\|\pi_{M^*}(x)\|^2}\pi_{M^*} (x)=
$$
$$
=\grad f(x)-\frac{\langle \grad f(x),\grad g(x)\rangle}{\|\grad g(x)\|^2}\grad g(x)-\frac{\langle \grad f(x),\pi_{M^*}(x)\rangle}{\|\pi_{M^*}(x)\|^2}\pi_{M^*} (x)=
$$
$$
=\grad f(x)+\dfrac{\langle\grad g(x),x\rangle\langle\grad g(x),\grad f(x)\rangle-\|\grad g(x)\|^2 \langle x,\grad f(x)\rangle}{\|x\|^2 \|\grad g(x)\|^2-\langle\grad g(x),x\rangle^2}x+
$$
$$
+\frac{\langle \grad f(x),\|\grad g(x)\|^2 x-\langle \grad g(x),x\rangle \grad g(x)\rangle \langle\grad g(x),x\rangle}{\|\grad g(x)\|^2(\|x\|^2 \|\grad g(x)\|^2-\langle\grad g(x),x\rangle^2)}\grad g(x)+
$$
$$
-\frac{ \langle\grad f(x),\grad g(x)\rangle(\|x\|^2 \|\grad g(x)\|^2-\langle\grad g(x),x\rangle^2)}{\|\grad g(x)\|^2(\|x\|^2 \|\grad g(x)\|^2-\langle\grad g(x),x\rangle^2)}\grad g(x)=v(x).$$
which proves (\ref{eq. v}).

For $x\in M^*$ we have
\[
\begin{split}
&\langle v(x),\grad f_M(x)\rangle=0\\
& \Leftrightarrow \langle \pi_{(\pi_{{M}^*}(x))^{\perp}}(\grad f_M(x)),\grad f_M(x)\rangle=0\\ 
&\Leftrightarrow \|\grad f_M (x)\|^2 \|\pi_{M^*} (x)\|^2 =\langle \grad f_M(x),\pi_{M^*} (x)\rangle^2
\\
&\Rightarrow x\in\mathcal{M}_0(f_M)
\end{split}
\]
which completes the proof.
\end{proof}

\begin{remark}
Analogously as in the proof of Lemma \ref{lem. wlasnosci v} we can prove that $v(x)=\pi_{(\pi_{\partial B} (\grad g(x)))^{\perp}} (\pi_{\partial B} (\grad f(x))$, where
\[\begin{split}
\pi_{\partial B} (w(x))&:=w(x)-\frac{\langle w(x),x\rangle}{\|x\|^2}x,\\
\pi_{(\pi_{\partial B} (\grad g(x)))^{\perp}}(w(x))&:=w(x)-\frac{\langle w(x),\pi_{\partial B} (\grad g(x))\rangle}{\|\pi_{\partial B} (\grad g(x))\|^2}\pi_{\partial B}(\grad g(x)) 
\end{split}
\]
for $x\in M^*,w\in \Omega$.
\end{remark}


We are ready to prove Theorem \ref{tw. trywializacja war. N}.

\begin{proof}
By the assumption that $\lambda$ is a $\rho_0$-regular value of $f_M$ at infinity, there exist a neighborhood $U$ of the $\lambda$ and $R>0$ such that
$$(\mathcal{M}_0(f_M)\cap f_M ^{-1}(U) )\backslash \overline{B(R)}=\emptyset,$$
where $\overline{B(R)}:=\{x\in M|\; R\geqslant\|x\|\}$. Let $w(x):=\frac{v(x)}{\langle v(x),f_M(x)\rangle}$ for $x\in f_M ^{-1} (U)\backslash \overline{B(R)}$. From the assumption  and Lemma \ref{lem. wlasnosci v}, $w$ is well defined. For each $\mu\in U$ consider the following system of differential equations
\begin{equation}\label{system mu2}
x'= (\lambda - \mu)w(x)
\end{equation}
with the right hand side defined in the  
set $G:=\{(t,x)\in \er\times M|\; x\in f^{-1}_M (U)\backslash \overline{B(R)}\}.$
Denote by $\Phi_\mu : V_\mu \to M $ the general solution of system (\ref{system mu2}) and  $V_\mu :=\{ (\tau,\eta,t) \in \er\times M\times R|\; (\tau,\eta) \in G , t\in I_\mu (\tau,\eta)\}$, where $I_\mu (\tau,\eta)$ is a domain of integral solution of $ t\to\Phi_\mu (\tau,\eta,t)$. From a definition of the general solution  we get 
\begin{equation}\label{eq. war poczatkowy 15}
\Phi_\mu (\tau,\eta,\tau)=\eta.
\end{equation} 
Note that Property \ref{propertyofv} implies
\begin{equation}\label{eq. **}
|\Phi_\mu (\tau,\eta,t)\|=\|\eta\| \text{ for } t\in I_\mu (\tau,\eta),(\tau,\eta)\in G,\mu \in U.
\end{equation}

Consider the mapping 
$$ \Psi_1: f^{-1}_M (U)\backslash \overline{B(R)} \ni x \mapsto \Phi_{f(x)} (0,x,1)\in f^{-1}_M (\lambda)\backslash \overline{B(R)}.$$
We show the mapping $ \Psi_1$ is well defined that is $ 1\in I_{f(x)} (0,x)$ for each $x\in f^{-1}_M (U)\backslash \overline{B(R)}$.

Suppose the contrary that there exists $x\in f^{-1}_M (U)\backslash \overline{B(R)}$ such that $1\notin I_{f(x)} (0,x)$  
and denote $\varphi_x (t):= \Phi_{f(x)} (0,x,t)$  for $ t\in I_{f(x)} (0,x).$ From (\ref{system mu2}) and the initial condition (\ref{eq. war poczatkowy 15}) we get 
\begin{equation}\label{eq. ***}
f_M \circ \varphi_x (t)=(\lambda-f(x))t+f(x),\quad t \in I_{f(x)} (0,x).
\end{equation}
Denoting $J$ as closed interval with endpoints $\lambda$ and $f(x)$ we see that the set
$$
F:=\{(t,x')\in \er\times f_M ^{-1}(U)\backslash \overline{B(R)}\; |\; t\in [0,1], f_M (x')\in J, \|x'\|=\|x\|\}.
$$
is a compact subset of $G$ such that the graph of ${\varphi_x |}_{ [0,\beta)}$ is contained in $F$. This contradicts Lemma \ref{lem.uciekanie} and proves $[0,1]\subset I_{f(x)} (0,x)$.

Using (\ref{eq. ***}) we get $f(\Psi_1(x))=f(\varphi_x(1))=\lambda$ and from (\ref{eq. **}) we have $\Psi_1(x)\in f_M^{-1}(\lambda)\backslash \overline{B(R)}$. Summing up we show that the mapping $\Psi_1$ is defined correctly. 
Similarly we can show that the mapping 
$$\Theta: f_M ^{-1} (\lambda)\backslash \overline{B(R)}\times U \ni (\xi,\mu)\mapsto \Phi_\mu (1,\xi,0)\in f_M ^{-1} (U)\backslash \overline{B(R)}$$
 is well defined. It is easy to check that the mapping 
$$\Psi: f_M ^{-1} (U) \backslash \overline{B(R)}\ni x \mapsto (\Psi_1 (x),f_M(x))\in f_M ^{-1} (\lambda)\backslash \overline{B(R)}\times U$$  
is a $C^{\infty}$ diffeomorphism and $\Psi^{-1} =\Theta$. Therefore ${f_M}_{|f_M ^{-1}(U)\backslash \overline{B(R)}}$
is a $C^{\infty}$ trivial fibration on $M$.
\end{proof}

\begin{remark}\label{rho}

From Theorem \ref{tw. trywializacja war. N} we have $B_{\infty}(f)\subset S_0 (f).$

\end{remark}

It is worth noting that unlike in a flat case we need to take into account the set $V:=\{x\in M|\;\exists_{t\in\er}\; \grad g(x)=tx\}\subset \mathcal{M}_0(f_M)$. In general the field $v:M\backslash V\to\er^n$ can not be continuously extended on the set $V$. The following example illustrates the fact

\begin{example}
 Let 
$$
g(x,y,z):=\dfrac{1}{2}x^2+y^2-\dfrac{1}{2},\quad f(x,y,z)=y\quad \text{for } (x,y,z)\in \er^3.
$$
We have
\[
\begin{split}
&v(x,y,z)=[v^1(x,y,z),v^2(x,y,z),v^3(x,y,z)]\\
&=\left[\frac{-2xyz^2}{x^2 y^2+x^2 z^2+4y^2 z^2},\frac{x^2 z^2}{x^2y^2+x^2z^2+4y^2z^2},\frac{x^2yz}{x^2y^2+x^2z^2+4y^2z^2}\right]
\end{split}
\]
 for $(x,y,z)\in M\backslash V$. 
  Note that $(1,0,0)\in V$ and
$$v^2 (1,0,z)=1 \text{ for } z\neq 0 \text{ and } v^2 (x,y,0)=0 \text{ for } x\neq 0 , y\neq 0,$$
therefore the limit $\lim_{(x,y,z)\to (1,0,0)} v^2(x,y,z)$ does not exist.
\end{example}

\medskip
{\bf Acknowledgement.} I would like to thank Stanisław Spodzieja for many conversations and valuable advice.

\bigskip

\noindent Michał Klepczarek\newline 
Faculty of Mathematics and Computer Science, University of \L \'od\'z, \newline
S. Banacha 22, 90-238 \L \'od\'z, POLAND \newline
 \indent E-mail: mikimiki88@tlen.pl

\end{document}